\documentclass[12pt]{article}

\def\gcheck{ {\bf [GARETH?]}}
\def\gcheck{}

\usepackage{amsmath}
\usepackage{amsfonts}
\usepackage{graphicx}
\usepackage{indentfirst}
\usepackage[font=bf]{caption}
\usepackage[boxed]{algorithm}
\usepackage{algorithmic}

\makeatletter
\newcommand\makebig[2]{%
  \@xp\newcommand\@xp*\csname#1\endcsname{\bBigg@{#2}}%
  \@xp\newcommand\@xp*\csname#1l\endcsname{\@xp\mathopen\csname#1\endcsname}%
  \@xp\newcommand\@xp*\csname#1r\endcsname{\@xp\mathclose\csname#1\endcsname}%
}
\makeatother

\makebig{biggg} {3.0}
\makebig{Biggg} {3.5}
\makebig{bigggg}{4.0}
\makebig{Bigggg}{4.5}

\usepackage[affil-it]{authblk}

\def\email#1{{\tt #1}}

\newtheorem{theorem}{Theorem}
\newtheorem{corollary}[theorem]{Corollary}

\newtheorem{proposition}[theorem]{Proposition}
\newtheorem{remark}{Remark}
\newenvironment{proof}{\paragraph{Proof:}}{\hfill$\square$}

\def\P{{\bf P}}
\def\E{{\bf E}}
\def\eqref#1{(\ref{#1})}
\def\text#1{{\rm #1}}
\def\Var{{\rm Var}}
\def\one{\textbf{1}}

\def\ddx{{d \over dx}}
\def\half{{1 \over 2}}
\def\sign{{\rm sign}}
\def\betamax{\beta_{max}^{(d)}}
\def\betamaxsimple{\beta_{max}}
\def\betamin{\beta_{min}}
\def\wmax{w_{max}^{(d)}}
\def\wmin{w_{min}^{(d)}}
\def\D{\mathcal{D}}
\def\gbar{\overline{g}}
\def\IN{{\bf N}}
\def\IZ{{\bf Z}}
\def\IR{{\bf R}}
\def\X{{\cal X}}

\def\Binomial{{\rm Binomial}}
\def\Binnh{\Binomial(n,1/2)}

\def\n2{{n \over 2}}
\def\m2{{m \over 2}}
\def\y2{{y \over 2}}
\def\z2{{z \over 2}}

\def\newI{\hat{I}}
\def\newI{I_*}
\def\curj{{j_*}}
\def\curj{{\overline{j}}}

\begin{document}

\title{
Skew Brownian Motion and
\\
Complexity of the ALPS Algorithm
}

\author{Gareth O.\ Roberts}
\affil{Department of Statistics\\
		University of Warwick\\
		United Kingdom\\
		CV4 7AL\\
              \email{Gareth.O.Roberts@warwick.ac.uk} }

\author{Je{f}frey S.\ Rosenthal}
\affil{Department of Statistical Sciences\\
	University of Toronto\\
	100 St. George Street, Room 6018\\
	Toronto, Ontario\\
	Canada M5S 3G3\\
	  \email{jeff@math.toronto.edu} }

\author{Nicholas G. Tawn}
\affil{Department of Statistics\\
		University of Warwick\\
		United Kingdom\\
		CV4 7AL\\
              \email{n.tawn.1@warwick.ac.uk} }

\date{(September 2020; last revised April 2021)}

\maketitle

\begin{abstract}

\noindent
Simulated tempering is a popular method of allowing MCMC algorithms to
move between modes of a multimodal target density $\pi$.
The paper~\cite{TawnEtAl2020} introduced the
Annealed Leap-Point Sampler (ALPS) to allow for rapid movement between modes.
In this paper, we prove
that, under appropriate
assumptions, a suitably scaled version of the ALPS algorithm converges
weakly to skew Brownian motion.
Our results show that under appropriate assumptions,
the ALPS algorithm mixes in time $O(d [\log d]^2)$
or $O(d)$, depending on which version is used.

\end{abstract}

\section{Introduction}

Markov chain Monte Carlo (MCMC) algorithms~\cite{handbook}
are very widely used to explore and sample from a
complicated high-dimensional target probability distribution $\pi$.
The most basic version of MCMC is the {\it Metropolis
algorithm}~\cite{Metropolis1953,hastings1970monte}.  From a given state~$x$,
it proceeds by first {\it proposing} to move to a new state~$y$,
and then either {\it accepting} that proposal (i.e., moving to~$y$),
or {\it rejecting} that proposal (i.e., staying at~$x$).  The {\it
acceptance probability} is given by $\min[1, \, \pi(y) \, / \, \pi(x)]$.
If the proposal densities are
symmetric (i.e., have the same probability of proposing~$y$ from~$x$,
as of proposing~$x$ from~$y$), this procedure ensures that the resulting
Markov chain will be reversible with respect to $\pi$, and thus have
$\pi$ as its stationary density.

MCMC algorithms have a tendency to get stuck
in local modes, which limits their effectiveness.  Annealing and tempering
methods~\cite{pincus, kirkpatrick1983optimization,aarts1988simulated,
geyer1991markov,marinari1992simulated} attempt to overcome this problem
by considering different powers $\pi^\beta$ of the target density, where
$\beta \le 1$ is an {\it inverse-temperature}.  Here $\beta=1$ corresponds
to the desired distribution, so those are the only samples which are
``counted''.  However, small positive values $\beta \ll 1$
make the density flatter and thus much easier to traverse.

Despite the tremendous success of tempering, these methods suffer from
deficiencies, especially in high dimensions. In particular, tempering
of distributions does not usually preserve the relative mass contained
in each of the modes. To deal with this, the paper \cite{TawnRR2020}
introduced a {\it weight-preserving transformation}
which overcomes the weight instability problem
as long as all modes look reasonably Gaussian.  Unfortunately, in
applications that is often not the case, since modes often exhibit
significant skewness.


An alternative approach, the Annealed Leap-Point Sampler (ALPS),
was introduced in \cite{TawnEtAl2020}.  This algorithm instead
considers very {\it large} values $\beta \gg 1$, corresponding to very
peaked target densities at very cold temperatures.  (Large $\beta$
are often used in optimisation algorithms such as {\it simulated
annealing}~\cite{pincus,kirkpatrick1983optimization,aarts1988simulated},
but are not normally used by sampling algorithms.)  Assuming smoothness,
the resulting sharply
peaked modes then become approximately Gaussian, thus facilitating simpler ways
of moving between them.  Furthermore, a weight-preserving transformation
is performed to approximately preserve the probabilistic weight of each
peak upon tempering.


For any MCMC algorithm, an important question is how quickly it
converges to its stationary distribution $\pi$.  While there have
been many attempts to bound MCMC convergence times directly (see
e.g.~\cite{computsimple} and the references therein), much of the
effort has been focused on questions of {\it computational complexity},
i.e.\ how the algorithm's running time grows as a function of other
parameters (dimension, size of data, etc.).

One promising, though technically challenging, approach to
determining the computational complexity of Metropolis algorithms
is through the use of {\it diffusion limits} as the dimension $d\to\infty$.
Similar to how
symmetric random walk converges to Brownian motion under appropriate
rescaling, certain transformations of some Metropolis algorithm components
will converge to Langevin diffusions.  This was originally exploited
in~\cite{roberts1997weak,roberts1998optimal} to derive complexity
and optimality results for ordinary random-walk-based Metropolis
algorithms, and was later generalised to many other contexts
\cite{roberts2001optimal, beda:rose:2008,roberts2014minimising}.
Furthermore, the $d\to\infty$ limit of MCMC algorithms also
provides good approximate information
about processes of modest finite dimension; see e.g.\
\cite[Figure~4]{roberts2001optimal}.

In this paper, will shall apply the diffusion limits
methodology to a ``vanilla'' version of
the ALPS algorithm, to study its convergence complexity.
We will prove (Theorem~\ref{Theorem:skewBM}) that, under appropriate
assumptions, a suitably scaled version of this ALPS algorithm converges
to {\it skew Brownian motion} (cf.~\cite{Lejay2006}).
This limit will allow us to draw
conclusions about the computational complexity of our algorithm, and to
show (Corollaries~\ref{complexitycor} and~\ref{QuanTAcor})
that under appropriate assumptions, as the dimension $d\to\infty$
the vanilla ALPS algorithm mixes in time $O(d [\log d]^2)$
or $O(d)$ depending on which version is used.


These results show that ALPS converges
fairly quickly even in high dimension.
This complexity order is similar to those previously derived
for ordinary random-walk Metropolis~\cite{roberts1997weak} and for
Simulated Tempering~\cite{roberts2014minimising}, which were each
shown to converge to
dimension-free diffusions when sped up by a factor of $d$, thus showing
that their complexity is $O(d)$.  The difference is that those previous
results assumed an iid target of the
form~\eqref{eq:themixer} and~\eqref{iidass} with $J=1$, and assumed
immediate mixing between all modes at each $\beta$, so
it omitted the issue of moving between modes which often makes those
algorithms exponentially slow~\cite{woodard2009sufficient}.  By contrast,
the ALPS algorithm stores mode location information that is
used in a special mode-jumping move to converge
efficiently even when there are $J>1$ widely-separated modes,
as we now describe.
\gcheck

\section{The ALPS Algorithm}

Tempering methods for MCMC usually consider powers $\pi^\beta$ for
{\it small} values $\beta \in (0,1]$,
to make the target distribution flatter
and thus allow for easier mixing between modes.  By contrast, the
paper~\cite{TawnEtAl2020} introduced the {\it Annealed Leap-Point
Sampler} (ALPS) algorithm, which instead uses {\it large} values
$\beta \gg 1$, combined with locally
weight-preserving tempering distributions as
in~\cite{TawnRR2020} so the modes retain their relative masses.
These choices make the modes of $\pi$ even more
separated.  However, under certain smoothness and integrability
assumptions, they also make each mode appear
approximately Gaussian and hence similarly-shaped.
This allows for auxiliary ``mode-jumping'' Markov chain steps
which move effectively between the different modes when $\beta$ is large.
Then, as usual, only samples in the original temperature $\beta=1$ are
``counted'' as actual samples from $\pi$.


To illustrate the idea of this algorithm, consider the following
simple example in dimension $d=5$.
Suppose the target density $\pi$ on $\IR^5$ is a mixture
of two skew-normal modes centered at $(-20,-20,-20,-20,-20)$
and $(20,20,20,20,20)$ respectively,
with scalings~1 and~2 respectively, and with shape parameter $\alpha=10$,
so for all $\theta\in\IR^5$,
$$
\pi(\theta) \ = \
(0.7) \prod_{i=1}^5 2 \, \phi\big( \theta_i+20 \big)
					\, \Phi\big( 10(\theta_i+20) \big)
\ + \
(0.3) \prod_{i=1}^5 \phi\big( \half(\theta_i-20) \big)
					\, \Phi\big( 5(\theta_i-20) \big)
\, ,
$$
where as usual $\phi(x) = {1 \over \sqrt{2\pi}} e^{-x^2/2}$
and $\Phi(x) = \int_{-\infty}^x \phi(u) \, du$; see Figure~\ref{fig-target}.

\begin{figure}[ht]
\centerline{\includegraphics[width=14cm]{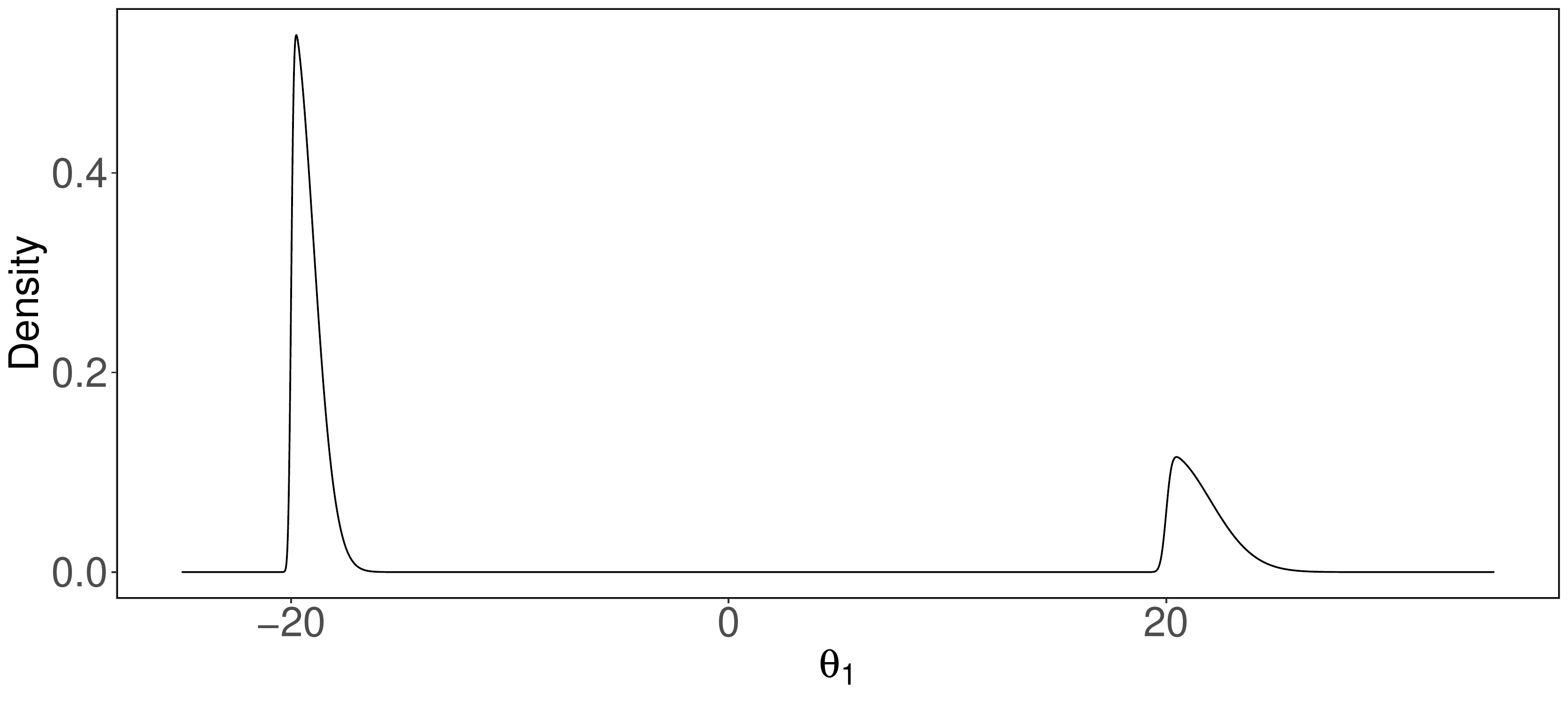}}
\caption{
The $\theta_1$ marginal of the target density in the illustrative example.}
\label{fig-target}
\end{figure}

In such an example, it is very easy for a Markov chain to mix separately
{\it within}
either of the two modes.  The challenge is to move between the modes
(which is virtually impossible for a typical fixed-temperature Metropolis
algorithm even in this simple 5-dimensional example).
The ALPS algorithm introduces a powerful independence-sampler-based
move so that at very large
inverse-temperature values $\beta \gg 1$, the chain can exploit the
near-Gaussianity of each of the modes to directly jump between them.
Figure~\ref{fig-betacolour} shows a trace plot of the
inverse-temperature values $\beta$ during one run of the algorithm, and also
indicates by colour which of the two modes the chain is in (i.e.,
closest to).  As can be seen from the plot, the chain
stays in the same mode for long periods of time, and only switches modes
when the values of $\beta$ are very large at which point it jumps to
either mode with its correct probability.
(Note that this description is for
the ``vanilla'' version of ALPS; see Remark~\ref{idealisedremark} below.)

\begin{figure}[ht]
\centerline{\includegraphics[width=16cm]{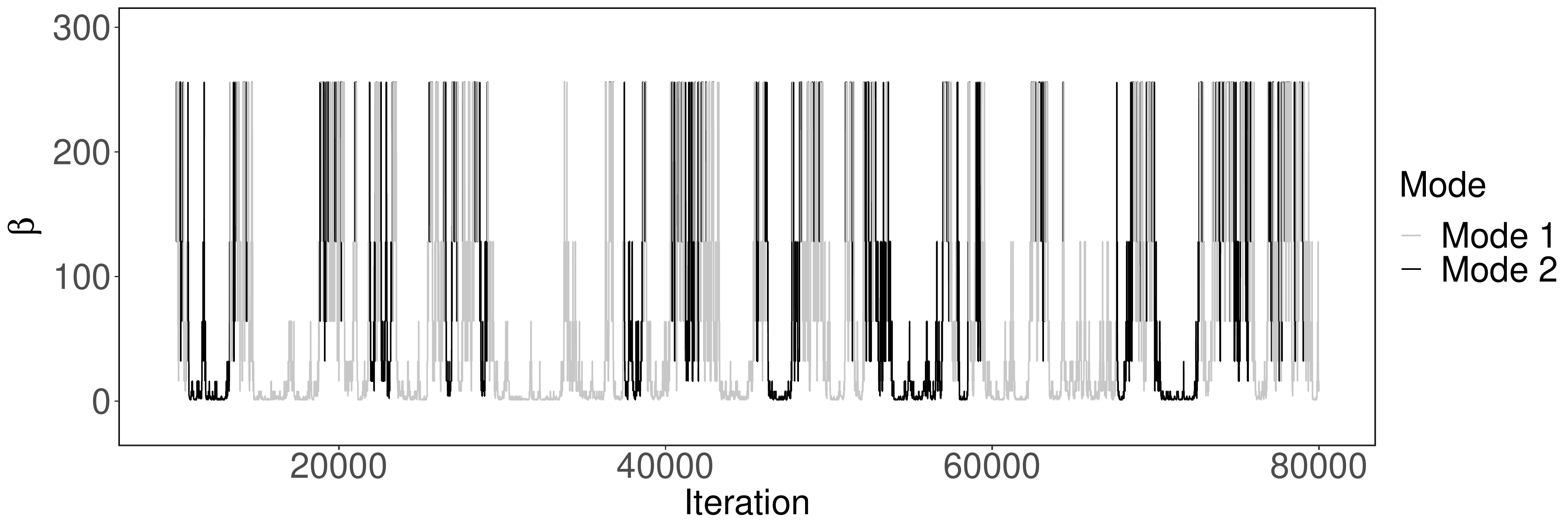}}
\caption{
Trace plot of the $\beta$ values in the illustrative example,
coloured to indicate whether the chain is in
mode~1 (light-gray) or mode~2 (black).
}
\label{fig-betacolour}
\end{figure}

Figure~\ref{fig-betacolour} illustrates that the key to the ALPS
algorithm's success is moving rapidly between the large $\beta
= \betamaxsimple = 256$ values
(which allow for mixing between the modes) and the small $\beta=1$ value
(which can be ``counted'' as a sample from $\pi$).  However, it is not
clear how quickly such mixing takes place, and in particular how it
changes depending on the target $\pi$ and dimension $d$.  To study this,
we would like to prove a diffusion limit of a suitably scaled version of
the $\beta$ process, but it is not clear from
Figure~\ref{fig-betacolour} what sort of limiting diffusive behaviour is
available.

To better understand this algorithm's convergence,
we consider a suitable transformation of $\beta$.
Namely, we instead consider the values of
$s \, \log(\betamaxsimple/\beta)$, where
$s=1$ if the chain is in mode~1 or $s=-1$ if the chain is in mode~2.
The resulting process is shown in Figure~\ref{fig-functional}, which
suggests that this modified functional does indeed start to resemble a
diffusive process.  Indeed, away from the special
value~0 (corresponding to $\betamaxsimple$ and the mode-jumping moves),
the process looks roughly like Brownian motion.
In fact, we shall prove below (Theorem~\ref{Theorem:skewBM})
that under appropriate assumptions and scalings, this modified process
converges to a skew Brownian motion.

\begin{figure}[ht]
\centerline{\includegraphics[width=16cm]{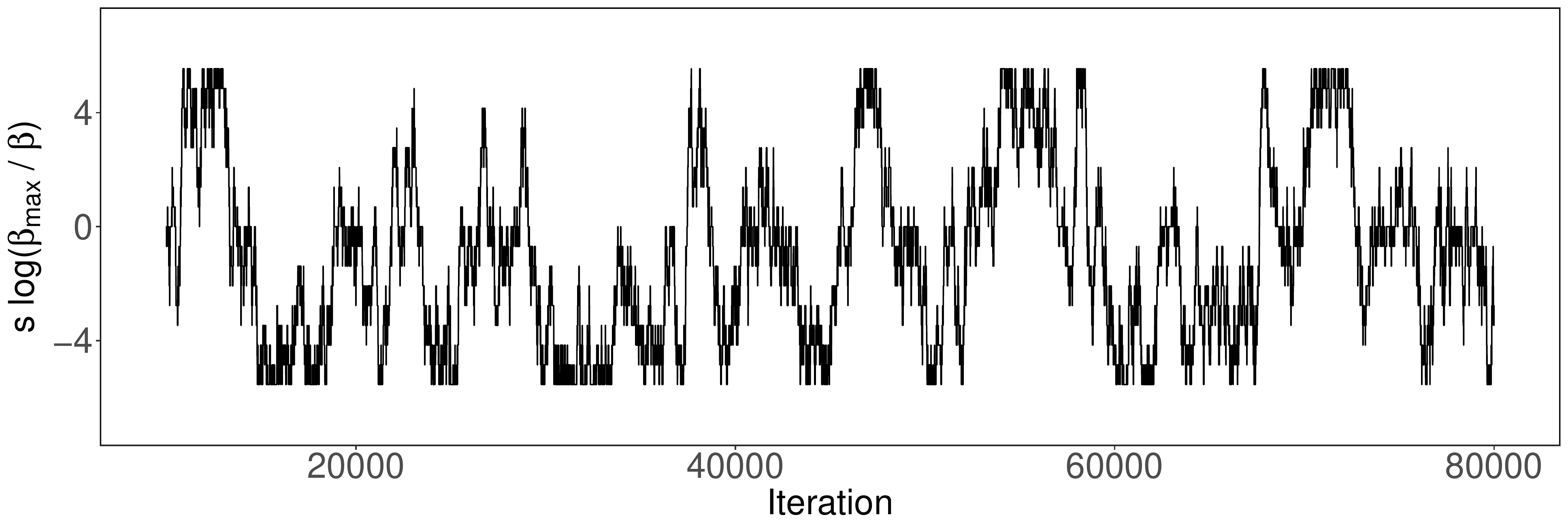}}
\caption{
A trace plot of the transformed values
$s \, \log(\betamaxsimple/\beta)$ in the illustrative example,
where $s=+1$ or $s=-1$ when the chain is in mode~1 or~2.
}
\label{fig-functional}
\end{figure}

More precisely, we shall prove diffusion limits for suitably rescaled
versions of the ALPS algorithm, as the dimension $d\to\infty$.  We shall
assume that the ALPS algorithm can easily jump between modes when it
reaches the sufficiently large inverse-temperature $\betamax$, but that
it is stuck within one mode whenever $\beta < \betamax$.  We therefore
focus on how the inverse-temperatures $\beta$ themselves are updated by
the algorithm.  In particular, we will prove (Theorem~\ref{Theorem:skewBM})
that a particular rescaling of the $\beta$ process converges to {\it
skew Brownian motion}~\cite{Lejay2006}.  This will in turn allow us to
derive computational complexity results (Section~\ref{sec-complexity}).


\begin{remark}\label{idealisedremark}
{\rm
The ``vanilla''
ALPS algorithm studied herein differs in certain ways from the full ALPS
algorithm for actual applications in~\cite{TawnEtAl2020}.
For example, we assume the process mixes perfectly between modes when
$\beta=\betamax$ (due to near-Gaussianity and
the algorithm's auxiliary mode-jumping steps),
and not at all when $\beta<\betamax$, while the full algorithm
mixes better and better at higher $\beta$ values but never
perfectly.  Also, the full algorithm actually uses {\it parallel
tempering}, in which a separate chain is run at each temperature and their
values are swapped; the single $\beta$ process studied herein can then be
thought of as following which of the chains is currently carrying
state information between larger and smaller inverse-temperatures and thus
facilitating mixing (cf.\ Section~4 of \cite{atchade2011towards}).
Finally, the full ALPS algorithm in~\cite{TawnEtAl2020} also makes use of
the QuanTA transformation~\cite{Tawn2018a}, an additional
affine transformation to increase the efficiency
of the temperature-swap moves, which we omit
here; we discuss the effect of this extra
QuanTA transformation in Corollary~\ref{QuanTAcor} below.
}
\end{remark}


\section{Assumptions}

We consider a version of the Annealed Leap-Point Sampler (ALPS)
algorithm of~\cite{TawnEtAl2020}.  We assume the chain always mixes
immediately within each mode, but the chain can only jump between modes
when at the sufficiently cold inverse-temperature $\beta = \betamax$,
at which point it immediately jumps to any of its modes with the correct
probability weight.


\ifx34
Then for each inverse-temperature $\beta \ge 1$, we set
\begin{equation}
        \pi_\beta(x) \ \propto \ \sum_{j=1}^J f_j(x,\beta)
	\ = \ \sum_{j=1}^J W_{(j,\beta)}
			\frac{[g_j(x)]^\beta}{\int [g_j(x)]^\beta dx}
\label{eq:themixerbeta}
\end{equation}
for appropriate weights $W_{(j,\beta)}$.
\fi

To facilitate theoretical analysis,
we assume that the target density
$\pi$ is a mixture of $J$ normalised densities $g_1,\ldots,g_J$
on $\IR^d$ with weights $w_1,\ldots,w_J$, i.e.\
\begin{equation}\label{eq:themixer}
\pi(x)
\ = \ \sum_{j=1}^J w_j g_j(x)
\, , \quad x\in\IR^d \, .
\end{equation}
We do not require the $g_j$ to be unimodal, but we shall nevertheless
refer to them informally as the ``modal components'' or
``modes'' of $\pi$, with the intuition that it is
easy for MCMC to mix efficiently within each individual $g_j$ but difficult
for it to jump between the different $g_j$.

We also assume that each state $x$ is ``allocated'' to (i.e.\ is ``in'')
one of the modes (e.g.\ whichever one's center it is closest to), such
that the accept/reject probabilities when updating $\beta$ can be computed
using only the mode $g_j$ of the current state, rather than
the full density $\pi$.
(This corresponds to considering the $x$ values as elements of $(\IR^d)^J$,
with a different version of the state space $\IR^d$ for each of the $J$
modes; if the modes are well separated, then especially for large $\beta$
this will be a good approximation to the actual algorithm.)
\gcheck


Then, for each inverse-temperature $\beta \ge 1$, we shall use
the tempered distribution
\begin{equation}\label{tempereddist}
\pi_\beta(x)
\ \propto \
\sum_{j=1}^J w_j \ {[g_j(x)]^\beta \over \int [g_j(x)]^\beta dx}
\ =: \
\sum_{j=1}^J w_j \, g_j^\beta(x)
\, ,
\end{equation}
where $g_j^\beta$ are the normalised powers of the $g_j$.
We assume the same weights $w_j$ can be used
for each $\beta$ due to a weight-preserving transformation
as in~\cite{TawnRR2020}.


In terms of these assumptions, the vanilla ALPS algorithm as we shall
study it is defined by Algorithm~\ref{VALPS}.
\gcheck

\begin{algorithm}
\caption{The Vanilla ALPS Algorithm}
\label{VALPS}
\begin{algorithmic} 
\REQUIRE A mixture target distribution $\pi$ on $\IR^d$
as in~\eqref{eq:themixer}.
\REQUIRE A sequence
of inverse-temperatures $1 = \beta_0 < \beta_1 < \ldots < \beta_k =:
\betamaxsimple$.
\REQUIRE An initial state $X_0\in\X$ and inverse-temperature
$\beta(0) := \beta_{I(0)}$.
\FOR{$n=0,1,2,3,\ldots$}

\STATE {\bf \# State-Changing Phase:}
\IF{$\beta(n+1)=\betamaxsimple$}
\STATE {\bf Sample} $\curj \in \{1,2,\ldots,J\}$ with probabilities $w_j$.
			\quad {\bf (auxiliary mode-jumping)}
\ELSE
\STATE {\bf Let} $\curj$ denote the mode that the current state $X_n$ is in.
\ENDIF
\STATE {\bf Sample} $X_{n+1} \sim g_\curj$.
		\quad {\bf (Mix immediately within the current mode only.)}

\STATE {\bf \# Temperature-Changing Phase:}
\STATE {\bf Select} a proposed new inverse-temperature $\beta_{\newI}
=\beta_{I(n)\pm 1}$ with probability 1/2 each.
\STATE {\bf Set} $\alpha \leftarrow \min\Big[1, \, {g_\curj^{\beta_{\newI}}(X_n)
			\over g_\curj^{\beta_{I(n)}}(X_n)}\Big]$,
with $g_\curj^\beta$ as in~\eqref{tempereddist}.
\ \ (Take $\alpha=0$ if $\newI=0$ or $\newI=k+1$).
\STATE {\bf With probability} $\alpha$,
{accept} the proposal by setting $I(n+1)=\newI$,
\STATE \quad {\bf Else},
{reject} the proposal by setting $I(n+1)=I(n)$.
\STATE {\bf Set} $\beta(n+1) \leftarrow \beta_{I(n+1)}$.

\ENDFOR
\end{algorithmic}
\end{algorithm}

In our theoretical proofs below, we assume for simplicity
(though see Remark~\ref{MoreThanTwoRemark} below)
that we have just $J=2$ modes, of weights $w_1$ and $w_2=1-w_1$ respectively.
To achieve limiting diffusions, we further assume
as in the original MCMC diffusion limit results~\cite{roberts1997weak}
that each of the individual components $g_j$
consists of iid univariate coordinates, i.e.\ that for each $j$
we have
\begin{equation}\label{iidass}
g_j(x) \ = \  \prod_{i=1}^d \gbar_j(x_i)
\end{equation}
for some fixed one-dimensional density function $\gbar_j$,
where $x=(x_1,x_2,\ldots,x_d)$.
This allows us to apply the diffusion-limit results
of~\cite{roberts2014minimising} within each individual target mode.
(Although~\eqref{iidass}
is a very restrictive assumption, it is known~\cite{roberts2001optimal}
that conclusions drawn from this special case are often approximately
applicable in much broader contexts.)


We also require assumptions on the $\beta$ values.
Write the inverse-temperatures as
\begin{equation}\label{betalist}
1 = \beta_0^{(d)} < \beta_1^{(d)} <
\ldots < \beta^{(d)}_{k(d)} =: \betamax
\end{equation}
for the process in dimension $d$.
Similar to \cite{atchade2011towards} and \cite{roberts2014minimising},
following \cite{predescu2004incomplete} and \cite{kone2005selection},
we assume that the inverse temperatures are related by
\begin{equation}
\label{eqn:beta}
\beta_i \, = \, \beta_{i-1} +
\ell(\beta_{i-1})/d^{1/2}
\end{equation}
for some fixed $C^1$ function $\ell$.
It is shown in
\cite{atchade2011towards,roberts2014minimising} that in
the single-mode iid case,
the fastest limiting diffusion is obtained by using the choice
\begin{equation}
\label{eqn:ell}
\ell(\beta) \ = \ I^{-1/2}(\beta) \, \ell_0
\end{equation}
for a fixed constant $\ell_0 \doteq 2.38$, where
$I(\beta) = \Var_{x \sim \gbar^\beta}(\log \gbar(x))$.
Inspired by this, in our later results we shall
assume the {\it Proportionality Condition} that the quantities
$I_j(\beta) := \Var_{x \sim \gbar_j^\beta}(\log \gbar_j(x))$
for the different modes are proportional, i.e.\
there are positive constants $r_j$ and a
$C^1$ function $I_0:\mathbb{R}_{+}\to\mathbb{R}_{+}$ such that
\begin{equation}\label{propcond}
I_j(\beta) \ = \ I_0(\beta)/r_j \, , \qquad j=1,\ldots,J
\, ,
\end{equation}
and shall then correspondingly assume that
\begin{equation}\label{propelleqn}
\ell(\beta) \ = \ I_0^{-1/2}(\beta) \, \ell_0
\, ,
\end{equation}
for some fixed constant $\ell_0>0$.

One example is the {\it Exponential Power Family}\/
case, in which each of the mixture component factors $g_j$
is of the form $g_j(x) \propto e^{-\lambda_j|x|^{r_j}}$ for some
$\lambda_j,r_j>0$.
It then follows from Section~2.4 of \cite{atchade2011towards} that
$I_j(\beta) = \beta^{-2}/r_j$ for $\beta>0$, so
the Proportionality Condition~\eqref{propcond} holds with
$I_0(\beta)=\beta^{-2}$.
The corresponding choice of $\ell$ from~\eqref{eqn:ell} is then
$\ell(\beta) = \beta/\sqrt{r_j}$ in mode~$j$.
This includes the Gaussian case, where each $r_j=2$
and $\lambda_j = 1/\sigma_j^2$.

\begin{remark}\label{ImmediateRemark}
{\rm
Our assumptions of immediate mixing within modes, and immediate mixing
between modes when $\beta=\betamaxsimple$, is analogous to the
corresponding assumptions in \cite[Section~5.1]{TawnRR2020} for ordinary
Simulated Tempering of immediate mixing within modes, and immediate mixing 
between modes when $\beta=\betamin$ (the hottest temperature).
In practice, even within simple modes the
mixing is not immediate, but rather takes e.g.\ $O(d)$ iterations
for random-walk Metropolis (RWM)~\cite{roberts1997weak}, or $O(d^{1/3})$
for Langevin algorithms~\cite{roberts1998optimal},
or $O(d^{1/4})$ for Hamiltonian (Hybrid) Monte Carlo~\cite{HMCtune}.
Since we shall show that the $\beta$-mixing for ALPS takes at least
$O(d)$, it follows that if Langevin or Hamiltonian dynamics are used for
the state-changing phase,
then the states will mix at a faster order than the temperatures,
thus effectively immediately,
in which case our assumption of immediate mixing within modes is reasonable.
By contrast, if RWM dynamics are used for the state-changing phase, then
the interplay between the temperature convergence
and state convergence would be more complicated, though it still works
effectively in practice~\cite{TawnEtAl2020}.
\gcheck
}
\end{remark}


\section{Main Results}
\label{sec-mainresults}

We now state various weak convergence results for various transformations
of our process.
(All proofs are deferred to Section~\ref{sec-proofs} below.)
Let $\beta^{(d)}(t)$ be the inverse temperature at time $t$
for the process described by Algorithm~\ref{VALPS} in dimension $d$.
Let $\beta^{(d)}(N(dt))$ be a continuous-time version of the
$\beta^{(d)}(t)$ process, sped up by a factor of $d$, where $\{N(t)\}$
is an independent standard rate-1 Poisson process.
To combine the two modes into one single process, we further augment
this process by multiplying it by $-1$ when the algorithm's state
is allocated to the second mode,
while leaving it positive (unchanged) when state is allocated
to the first mode, i.e.\
\begin{equation}\label{eq:xtddef}
X_t^{(d)} \ = \
\begin{cases}
 \beta^{(d)}({N(dt)}) \, , &\ \textrm{in mode~1} \\
 - \beta^{(d)}({N(dt)}) \, , &\ \textrm{in mode~2} \\
\end{cases}
\end{equation}
Our first diffusion limit result,
following \cite{roberts2014minimising}, states that within each mode,
the inverse temperature process
behaves identically to the case where there is only one mode (i.e.\ $J=1$).
To state it, we extend the definition of $I$ to
\begin{equation}\label{extendI}
I(\beta) \ = \
\begin{cases}
\text{Var}_{x\sim f_1^\beta}(\log f_1(x)) \, , &\ \beta>0 \\
\text{Var}_{x\sim f_2^{|\beta|}}(\log f_2(x)) \, , &\ \beta<0 \, . \\
\end{cases}
\end{equation}


\begin{theorem}\label{Theorem:Xdiffusion}
Assume the target distribution $\pi$ is of the form~\eqref{eq:themixer},
with $J=2$ modes of weights $w_1$ and $w_2=1-w_1$,
each having iid coordinates as in~\eqref{iidass},
with tempered distributions as in~\eqref{tempereddist} for
an inverse-temperatures list~\eqref{betalist} related by~\eqref{eqn:beta}.
Then away from its boundary points~1 and~$\betamax$,
the process $\{X_t^{(d)}\}$ from \eqref{eq:xtddef}
converges weakly as $d\to\infty$
to a fixed diffusion process $X$, which for $X^{(d)}>0$ satisfies
\begin{multline}\label{Xdifeqn}
dX_t  \ =\   \left[2  \ell^2(X_t)
  \, \Phi\left( {- \ell(X_t) I^{1/2}(X_t) \over 2} \right) \right]^{1/2} dB_t
\big. \\
 \ \ + \, \Bigg[\ell(X_t) \, \ell'(X_t)
  \ \Phi \left({-I^{1/2}(X_t) \ell(X_t) \over 2}\right)
\big. \\
 \ \ - \, \ell^2(X_t) \left({\ell(X_t) I^{1/2}(X_t)\over 2} \right)'
	\phi \left( {-I^{1/2}(X_t) \ell(X_t) \over 2} \right) \Bigg] dt
\, .
\end{multline}
The same equation holds for $X_t<0$,
except with the sign of the drift reversed.
\end{theorem}

As a check, \eqref{Xdifeqn} satisfies the general relation $\mu(x)
= \half \sigma^2(x) \ddx \log \pi(x) + \sigma(x) \sigma'(x)$,
which implies that $\pi$ is {\it locally invariant} for $X^{(d)}$,
i.e.\ that its generator $G$ has $\pi \, (G f)(x) = 0$
for appropriate smooth $f$ and for $x$ in the interior of the domain,
That is, $\pi$ is
stationary for $X^{(d)}$ locally within each mode, as expected.

However, Theorem~\ref{Theorem:Xdiffusion} describes only what
happens on each mode separately; it says nothing about the mode-jumphing
process itself.  Moreover, its state space $(-\infty,-1]\cup [1,\infty)$
is not connected.  In fact,
we will see below that as $d\to\infty$,
the value $\betamax$ will go to infinity and
hence never be reached in finite time.
To resolve these issues, we make several
transformations on the $X^{(d)}_t$ process.
First, for $|x| \ge 1$, we define
$$
h(x) \ = \
\int_1^{|x|} {1\over \ell (u)} du
\, .
$$
(For example, in the Exponential Power Family case,
$I(\beta) \propto 1/\beta^2$, so~\eqref{propelleqn} gives
$\ell(\beta) = I_0^{-1/2}(\beta) \, \ell_0 \propto \beta$,
whence
$
h(x) \, = \, \int_1^{|x|} {1 \over \ell(u)} \, du
\, \propto \, \int_1^{|x|} {1 \over u} \, du
\, = \, \log|x|
$.)
We then set
\begin{equation}\label{Htransform}
H_{t \, (h(\betamax))^2}^{(d)}
\ = \
\sign\left(X_{t \, h(\betamax)^2}^{(d)}\right) \ \left[ 1 +
    {h\left( X^{(d)}_{t \, h(\betamax)^2} \right) \over h(\betamax)} \right]
\, .
\end{equation}
Hence, $1 \le H_t^{(d)} \le 2$ in the first mode,
and $-1 \ge H_t^{(d)} \ge -2$ in the second mode.
Also, $H_t^{(d)}$ speeds up $X_t^{(d)}$ by a factor of $h(\betamax)^2$,
and hence moves at Poisson rate $d \, h(\betamax)^2$.
This new process $H_t^{(d)}$ satisfies the following.

\begin{theorem}\label{Theorem:Hdiffusion}
Under the set-up and assumptions of Theorem~\ref{Theorem:Xdiffusion},
on $(-2,-1) \cup (1,2)$ (i.e., away from its boundary points),
the process $\{H_t^{(d)}\}$ from~\eqref{Htransform}
converges weakly in the Skorokhod topology
as $d\to\infty$ to a limiting diffusion $H$ which satisfies
\begin{equation}\label{dHt}
dH_t \ = \
\left[ 2 \, \Phi\left( {- \ell(X_t) I^{1/2}(X_t) \over 2} \right)
					\right]^{1/2} dB_t
+ \ \ell(X_t)
  \, \Bigg[ \Phi \left( {-I^{1/2}(X_t) \ell(X_t) \over 2} \right) \Bigg]' dt
\, .
\end{equation}
Furthermore, $H$ leaves constant (uniform) densities locally invariant.
\end{theorem}

To make further progress, we now use the Proportionality
Condition~\eqref{propcond},
with corresponding inverse-temperature spacing~\eqref{propelleqn}.
It then follows from the extended definition~\eqref{extendI} that
$\ell(X_t) \, I^{1/2}(X_t) = \ell_0 r_1^{1/2}$ for $X_t < 0$,
and $\ell(X_t) \, I^{1/2}(X_t) = \ell_0 r_2^{1/2}$ for $X_t > 0$,
with $\Big[ \ell(X_t) \, I^{1/2}(X_t) \Big]' = 0$ for all $X_t \not= 0$.
Hence, Theorem~\ref{Theorem:Hdiffusion} immediately gives:

\begin{corollary}\label{intBMcor}
Assume the set-up and assumptions of Theorem~\ref{Theorem:Xdiffusion},
and also the Proportionality Condition~\eqref{propcond}
with inverse-temperature spacing~\eqref{propelleqn}.
Then as $d\to\infty$, the process $\{H_t^{(d)}\}$ converges weakly in the
Skorokhod topology to a limit process $H$
on $(-2,-1)$ and on $(1,2)$, i.e.\ away from its boundary points.
Furthermore, $H$ is a diffusion, with drift~0, and with
diffusion coefficient which
is constant on each of the two intervals $(-2,-1)$ and $(1,2)$.
Specifically,
$$
dH_t \ = \ s(H_t) \ dB_t
\, ,
$$
where $s(H_t) = s_1$ for $H_t \in (1,2)$,
and $s(H_t) = s_2$ for $H_t \in (-2,-1)$, with
\begin{equation}\label{sdef}
s_i \ := \
\left[ 2 \, \Phi\left( - \half \ell_0 r_i^{1/2} \right) \right]^{1/2}
\, .
\end{equation}
\end{corollary}

Next, we need to join up the two parts of the domain $[-2,-1] \cup
[1,2]$ of the process $H_t^{(d)}$.
Now, the original process can jump
between modes when at the coldest temperture $\betamax$, corresponding
to the values $\pm 2$ for the transformed process $H_t^{(d)}$.  Hence, we let
\begin{equation}\label{Zdef}
Z_t^{(d)} \ = \ 
2 \ \sign(H_t^{(d)}) - H_t^{(d)} \ = \
\begin{cases}
2 - H_t^{(d)} \, , &\ H_t^{(d)} \ge 1, \ \textrm{i.e.\ in mode~1} \\
-2 - H_t^{(d)} \, , &\ H_t^{(d)} \le -1, \ \textrm{i.e.\ in mode~2} \\
\end{cases}
\end{equation}
so that $Z_t^{(d)}$ has domain $[-1,1]$ with mode-jumping at~0.


However, by Corollary~\ref{intBMcor}, the limit of the process
$Z_t^{(d)}$ will still have diffusion coefficient
$s_1$ or $s_2$ on its positive
and negative parts.  We thus rescale the process by setting
\begin{equation}\label{Wdef}
W^{(d)}_t \ = \ s(Z^{(d)}_t)^{-1} \, Z^{(d)}_t
\, .
\end{equation}
(So, to recap, $W_t^{(d)}$ is defined in~\eqref{Wdef} in terms of
$Z_t^{(d)}$, which is defined in~\eqref{Zdef} in terms of
$H_t^{(d)}$, which is defined in~\eqref{Htransform} in terms of
$X_t^{(d)}$, which is in turn defined in~\eqref{eq:xtddef} in terms
of the original inverse-temperature process $\beta^{(d)}(t)$,
which itself arises from running Algorithm~\ref{VALPS}.)
Then $W^{(d)}_t$ has domain $[-{1 \over s_2},{1 \over s_1}]$,
and limit which is
actual Brownian motion on each of $(-{1 \over s_2},0)$ and $(0,{1 \over s_1})$.
The precise limit of this process requires the notion of {\it skew
Brownian motion}, a generalisation of usual Brownian motion that,
intuitively, behaves just like a Brownian motion except that the sign
of each excursion from~0 is chosen using an independent Bernoulli random
variable; for further details and constructions and discussion see e.g.\
\cite{Lejay2006}.  In terms of skew Brownian motion, we have:

\begin{theorem}\label{Theorem:skewBM}
Under the assumptions of Corollary~\ref{intBMcor},
with $s_i$ as in~\eqref{sdef},
the process $\{W_t^{(d)}\}$ from~\eqref{Wdef} converges weakly in the
Skorokhod topology as $d\to\infty$
to a limit process $W$ which is skew Brownian
motion on $[-{1 \over s_2},{1 \over s_1}]$, with reflecting boundaries,
and with excursion probabilities
at~0 proportional to $w_1 s_1$ (to go positive) and $w_2 s_2$
(to go negative).
\end{theorem}

The above theorems are all proven in Section~\ref{sec-proofs} below.
First, we use them to investigate the computational
complexity of the ALPS algorithm.

\section{Computational Complexity}
\label{sec-complexity}

Theorem~\ref{Theorem:skewBM} above has implications for the computational
complexity of the ALPS algorithm.
Indeed, it shows that the limiting process $W$ does not depend at all on
the dimension $d$, and hence has convergence time $O(1)$ as $d\to\infty$.
However, $W$ was derived from the processes $H_t^{(d)}$ and $Z_t^{(d)}$,
which sped up time
by a factor of $(h(\betamax))^2$ from the process $X_t^{(d)}$, which itself
sped up time by a factor $d$.
That is, $W$ was sped up by a total factor of $d [h(\betamax)]^2$.
So, in the original scaling, the
convergence time is $O(d [h(\betamax)]^2)$.

More formally, it is shown in~\cite[Theorem~1]{RRcomplexity} that
such diffusion limit convergence implies that for any $\epsilon>0$,
the {\it convergence time} $T_\epsilon$ for each component of
the original process
to get within $\epsilon$ of stationary in Kantorovich-Rubinstein distance,
averaged over starting state chosen from stationarity, will be of the same
order as the speedup factor.  So, combining Theorem~\ref{Theorem:skewBM}
and~\cite[Theorem~1]{RRcomplexity} shows that for the $\beta$ process of
the vanilla ALPS algorithm, the convergence time
$T_\epsilon$ is $O(d [h(\betamax)]^2)$.
Furthermore, this convergence time is indeed an appropriate measure of the
algorithm's efficiency, since it is proportional
to the rate at which the $\beta$ values can complete a ``round
trip'' from one sample at $\beta=1$, to a mode jump at
$\beta=\betamaxsimple$, to another sample at $\beta=1$, and hence mix well
between modes; similar approaches
appear in \cite{atchade2011towards, Jacka2019, Syed2020}.
\gcheck

This raises the question of how $h(\betamax)$ grows as a function of $d$.
It is proven in~\cite{TawnEtAl2020} that for the ALPS
process to mix modes efficiently, we need the maximum inverse-temperature
value $\betamax$ to grow linearly with dimension, i.e.\ we need
to choose $\betamax \propto d$.
And, in the Exponential Power Family case, as mentioned above,
$I(\beta) \propto 1/\beta^2$ which implies by~\eqref{propelleqn} that
$h(x) \propto \log|x|$,
so $h(\betamax) \propto \log(d)$.  Hence, the complexity order
$O(d [h(\betamax)]^2)$ equals $O\big(d \, [\log d]^2\big)$.
That is, for the inverse temperature process to
hit $\betamax$ and hence mix modes takes
$O\big(d \, [\log d]^2\big)$ iterations.

If we are not in the Exponential Power Family case, then it may no longer
be true that $I(\beta) \propto 1/\beta^2$.  However,
as $d,\beta\to\infty$, under appropriate smoothness assumptions
the densities in the
different modes will become approximately Gaussian, which corresponds to
the Exponential Power Family case with $r=2$.
And, it is proven in equation~(66) of~\cite{Tawn2018a} that
if the first four moments converge to those of a Gaussian,
then $2\beta^2 I(\beta) \to 1$, i.e.\
approximately $I(\beta) \propto 1/\beta^2$.
Hence, from~\eqref{propelleqn},
approximately $\ell(\beta) \propto \beta$, so again
$h(\betamax) \propto \log(d)$, and the complexity order is
still $O\big(d \, [\log d]^2\big)$ as before.
We summarise this conclusion as follows.


\begin{corollary}\label{complexitycor}
Under the assumptions of Corollary~\ref{intBMcor},
if either
{\bf (a)} the densities of the two modes of $\pi$ are in the
Exponential Power Family, or
{\bf (b)} the two modes' first
four moments each converge to those of a Gaussian as $d,\beta\to\infty$,
then
the convergence times $T_\epsilon$ for $\beta$ are
$O\big(d \, [\log d]^2\big)$ as $d\to\infty$.
\end{corollary}


In a different direction, the paper~\cite{Tawn2018a}
introduces a {\it QuanTA Algorithm}, which modifies parallel
tempering's usual temperature-swap moves by adjusting the $x$ space
in order to permit larger moves in the inverse temperature space.
As a result of this, they show~\cite[Theorem~2]{Tawn2018a} that the resulting
$\ell(\beta)$ function is then proportional to $\beta^{k/2}$ for some $k>2$
(instead of proportional to $\beta$).  In that case,
$$
h(\betamax)
\ = \ \int_1^{\betamax} {1 \over \ell(u)} \, du
\ \le \ \int_1^\infty {1 \over \ell(u)} \, du
\ \propto \ \int_1^\infty u^{-k/2} \, du
\ = \ (k/2) - 1
\ < \ \infty
\, ,
$$
so that $h(\betamax)$ is $O(1)$ rather than $O(\log d)$.
This means that the convergence complexity $O(d [h(\betamax]^2)$
becomes simply $O(d)$, i.e.\ the $[\log d]^2$ factor vanishes.  We
summarise this observation as follows.

\begin{corollary}\label{QuanTAcor}
Under the assumptions of Corollary~\ref{intBMcor},
if we instead run the version of the ALPS algorithm which uses
the QuanTA modification of~\cite{Tawn2018a}, then
the convergence times $T_\epsilon$ for $\beta$ are
$O\big(d \, [\log d]^2\big)$ as $d\to\infty$.
\end{corollary}

Comparing Corollaries~\ref{complexitycor} and~\ref{QuanTAcor}, we
see that the QuanTA modification improves the complexity bound by
a factor of $[\log d]^2$.  This is not surprising, since QuanTA was
specifically designed to make the algorithm move faster especially under
near-Gaussianity at large $\beta$, thus improving the mixing time.
This improvement is also borne out through simulation experiments;
see~\cite{Tawn2018a}.  \gcheck



\begin{remark} (More than Two Modes.)\label{MoreThanTwoRemark}
{\rm
For simplicity, all of the above proofs
assumed a mixture of just $J=2$ modes.
However, similar analysis works more
generally.  Indeed, suppose $\pi$ is a mixture of $J>2$ modes, of
weights $w_1,w_2,\ldots,w_J \ge 0$ where $\sum_{i=1}^J w_i = 1$.  Then when
$\beta(t)$ reaches $\betamax$, the process chooses one of the
$J$ modes with probability $w_i$
(due to the auxiliary mode-jumping step).
In this case, a theorem similar
to Theorem~\ref{Theorem:skewBM} could be proven by similar methods.
The processes $\{W_t^{(d)}\}$ will converge not to skew Brownian motion
but to {\it Walsh's Brownian motion}, a process
not on $[-{1\over s_2},{1 \over s_1}]$
but rather on a ``star''
shape with $J$ different line segments all meeting at the origin
(corresponding to $\betamax$).
Intuitively, this process
behaves as Brownian motion within each segment, but
chooses each excursion from the origin using an
independent random variable with probabilities $w_i$;
for further details and constructions and discussion see e.g.\
\cite{barlow1989}.
(The case $J=2$ but $w_1\not=1/2$ corresponds to skew Brownian
motion as in Theorem~\ref{Theorem:skewBM}.)
This in turn leads to the same complexity bound of
$O\big(d \, [\log d]^2\big)$ iterations
(or $O(d)$ iterations if using QuanTA) when $J>2$ as well.
}
\end{remark}

\section{Theorem Proofs}
\label{sec-proofs}

In this section, we prove the theorems stated in
Section~\ref{sec-mainresults} above.
Note that Theorem~\ref{Theorem:Xdiffusion}
essentially follows directly from previous theoretical
analysis of Simulated Tempering
in~\cite{atchade2011towards,roberts2014minimising},
and Theorem~\ref{Theorem:Hdiffusion} then
follows from some additional computations using Ito's Formula.
By contrast, Theorem~\ref{Theorem:skewBM} requires a different approach,
to show that the modified $W^{(d)}$ process converges
to reflecting skew Brownian motion,
including delicate arguments to show convergence of the
corresponding infinitesimal generators especially at the two
endpoints and at the excursion point~0.
\gcheck

\subsection{Proof of Theorem~\ref{Theorem:Xdiffusion}}

Since mixing between modes is only possible at $\betamax$, the
dynamics for other $\beta$
will be identical to the single mode case ($J=1$) as covered
in \cite{atchade2011towards,roberts2014minimising}.
It therefore follows directly from Theorem~6 of
\cite{roberts2014minimising} that as $d\to\infty$,
the process $\{X_t\}$ converges weakly, at least on $X_t>0$,
to a diffusion limit $\{X_t\}_{t \ge 0}$ satisfying~\eqref{Xdifeqn}.
The result for $X_t<0$ follows similarly.


\subsection{Proof of Theorem~\ref{Theorem:Hdiffusion}}

We assume $x\in(1,2)$; the proof for $x\in(-2,-1)$ is virtually identical.
Here $H_t = h(X_t)$, where
$h'(x) = \ell(x)^{-1}$, and $h''(x) = -\ell'(x) \ell(x)^{-2}$.
Hence, by Ito's Formula,
\begin{eqnarray*}
dH_t
&=& h'(X_t) dX_t + \frac{1}{2}  h''(X_t) d\langle X \rangle_t \\
&=& \ell(X_t)^{-1} dX_t - \frac{1}{2} \ell'(X_t) \ell(X_t)^{-2}
d\langle X \rangle_t \\
&=& \ell(X_t)^{-1}
\left[2  \ell^2(X_t)
   \Phi\left( \frac{- \ell(X_t) I^{1/2}(X_t)}{ 2} \right) \right]^{1/2} dB_t\\
 && + \ell(X_t)^{-1} \ell(X_t)  \ell'(X_t)
   \Phi \left(\frac{-I^{1/2}(X_t) \ell(X_t)}{ 2}\right) dt \\
 && - \ell^2(X_t) \left(\frac{\ell(X_t) I^{1/2}(X_t)}{ 2} \right)'
	\phi \left( \frac{-I^{1/2}(X_t) \ell(X_t)}{  2} \right) dt \\
 && - \frac{1}{2} \ell'(X_t) \ell(X_t)^{-2} 2 \ell^2(X_t)
   \Phi\left( \frac{- \ell(X_t) I^{1/2}(X_t)}{ 2} \right) dt
\end{eqnarray*}
In this last equation,
the second and fourth terms cancel.  Also, since $\Phi' = \phi$,
it follows from the chain rule that the third term can be written as
$$
- \ell^2(X_t)
\left[ \Phi \left( \frac{-I^{1/2}(X_t) \ell(X_t)}{  2} \right) \right]' dt 
\, .
$$
This gives~\eqref{dHt}.
Then, writing everything in terms of $H_t = h(X_t)$, this becomes
$$
dH_t
\ = \
\left[ 2 \, \Phi\left( {- \ell(h^{-1}(H_t)) I^{1/2}(h^{-1}(H_t)) \over 2}
	\right) \right]^{1/2} dB_t
$$
$$
  + \, \ell(h^{-1}(H_t))
  \, \Bigg[ \Phi \left( {-I^{1/2}(h^{-1}(H_t)) \ell(h^{-1}(H_t)) \over 2}
	\right) \Bigg]' dt
\, .
$$
Now, a diffusion of the form $dH_t = \sigma(H_t) dB_t +
\mu(H_t) dt$ has locally invariant distribution $\pi$ provided that
$\frac{1}{2} (\log\pi)' \sigma^2 + \sigma \sigma' = \mu$.  That holds
for constant $\pi$ if $\sigma \sigma' = \mu$.
In this case, we compute that
$$
\sigma \sigma' \ = \ \half (\sigma^2)'
\ = \ \half {d \over dH} \left[ 2 \, \Phi\left( {- \ell(h^{-1}(H))
I^{1/2}(h^{-1}(H)) \over 2} 
        \right) \right]
$$
$$
\ = \ \half \left( {dH \over dX} \right)^{-1}
{d \over dX} \left[ 2 \, \Phi\left( {- \ell(X) I^{1/2}(X) \over 2} 
        \right) \right]
$$
$$
\ = \ \half \left( \ell(X)^{-1} \right)^{-1}
\left[ 2 \, \Phi\left( {- \ell(X) I^{1/2}(X) \over 2} 
        \right) \right]'
$$
$$
\ = \ \ell(X) \left[ \Phi\left( {- \ell(X) I^{1/2}(X) \over 2} 
        \right) \right]'
\ = \ \mu
\, ,
$$
thus showing that $H$ leaves constant densities locally invariant.


\subsection{Proof of Theorem~\ref{Theorem:skewBM}}

Let $\wmin = -{1 \over s_2}$
and $\wmax = {1 \over s_1}$ be the endpoints of the domain of $W$.
By Corollary~\ref{intBMcor}, $dH_t =
s(H_t) \, dB_t$ in the interior of its domain.  Since $W_t = s(H_t)^{-1}
\, H_t$, it follows that $W_t$ behaves like
Brownian motion on $(-\wmin,0)$ and on $(0,\wmax)$.
It remains to show that the process converges weakly to skew Brownian
motion, including at the boundary points $W_t=0,\wmin,\wmax$.
We prove this result using infinitesimal generators, as we now explain.

\subsubsection{Method of Proof: Generators}

To prove the weak convergence, it suffices by Corollary~8.7 of Chapter~4 of
\cite{ethi:kurt:1986} to show (similar to previous proofs of diffusion
limits of MCMC algorithms in \cite{roberts1997weak, roberts1998optimal,
beda:rose:2008}) that the {\it infinitesimal generator} $G^{(d)}$ of the
process $W^{(d)}$ converges uniformly in~$x$ as $d\to\infty$
to the generator $G^*$ of skew Brownian motion,
when applied to a {\it core} $\D$ of functionals, i.e.\ that
$$
\lim_{d\to\infty}
\ \sup_{x \in [\wmin,\wmax]}
\ \left| G^{(d)}f(x) - G^*f(x) \right|
\ = \ 0
\, ,
\qquad f \in \D
\, ,
$$
where
$$
G^{(d)}f(x) \ := \
\lim_{\delta\searrow 0}
{\E[f(W^{(d)}_\delta) \, | \, W^{(d)}_0=x] - f(x) \over \delta}
\, .
$$


To this end, let $\D$ be the set of all functions
$f:[-\wmin,\wmax]\to \mathbb{R}$ which
are continuous and twice-continuously-differentiable on
$[\wmin,0]$ and also on $[0,\wmax]$,
with matching one-sided second derivatives $f''^+(0)=f''^-(0)$,
and skewed one-sided first derivatives satisfying
$w_1 s_1 f'^+(0)= w_2 s_2 f'^-(0)$,
and $f'(\wmax)=f'(\wmin)=0$.
Then it follows from e.g.\ \cite{liggett:2010}
and Exercise~1.23 of Chapter~VII of \cite{revuz:yor:2004})
that the generator of skew Brownian motion (with excursion weights
proportional to $w_1 s_1$ and $w_2 s_2$ respectively, and with reflections at
$\wmin$ and $\wmax$) satisfies that
$G^*f(x) = \frac{1}{2} f''(x)$ for all $f\in \D$,
where $f''(0)$ represents the common value $f''^+(0) = f''^-(0)$.
Furthermore,
$\D$ is clearly dense (in the sup norm) in the set of all
$C^2[\wmin,\wmax]$ functions, so in the language of
\cite{ethi:kurt:1986}, $\D$ serves as a core of functions for which
it suffices to prove that the generators converge.


It follows from Corollary~\ref{intBMcor}, as discussed above,
that for any fixed $f\in \D$,
\begin{equation}
\lim_{d\to\infty}
\sup_{w \in \left( \wmin,  \wmax \right) \setminus \{0\}}
|G^{(d)}f(w) - G^* f(w)| = 0
\, .
\label{eq:genconv}
\end{equation}
That is, the generators do converge uniformly to $G^*$, as required,
at least for $w \not= 0, \wmin, \wmax$, i.e.\ avoiding the mode-jumping
value~0 and the reflecting boundaries $\wmin$ and $\wmax$.
To complete the proof, it suffices to prove that
\eqref{eq:genconv} also holds at $w=0,\wmin,\wmax$, i.e.\ to prove
\begin{equation}\label{G0eqn}
\lim_{d\to\infty} G^{(d)}f(0) \ \equiv \ G^* f(0)
\ = \ \, \frac{1}{2} \, f''(0)
\, ,
\end{equation}
\begin{equation}\label{Gwmineqn}
\lim_{d\to\infty} G^{(d)}f(\wmin) \ \equiv \ G^* f(\wmin)
\ = \ \, \frac{1}{2} \, f''(\wmin)
\, ,
\end{equation}
and
\begin{equation}\label{Gwmaxeqn}
\lim_{d\to\infty} G^{(d)}f(\wmax) \ \equiv \ G^* f(\wmax)
\ = \ \, \frac{1}{2} \, f''(\wmax)
\, .
\end{equation}

\subsubsection{Verification of~\eqref{Gwmineqn} and~\eqref{Gwmaxeqn}}
\label{sec-1314}

The proofs of~\eqref{Gwmineqn} and~\eqref{Gwmaxeqn} are virtually
identical, so here we prove~\eqref{Gwmaxeqn}.

If the original inverse-temperature
process $\beta^{(d)}(t)$ proposes to move in time~1 from
inverse-temperature $1+0=1$ to $1 + \ell(1) d^{-1/2}$,
then by~\eqref{Htransform}, the $H^{(d)}_t$ process proposes to move
at Poisson rate~$[d \, h(\betamax)^2]$ from $1 + {0 \over h(\betamax)} = 1$ to
$$
1 + {h\big(1+\ell(1) d^{-1/2}\big) \over h(\betamax)}
\ = \
1 + {1 \over h(\betamax)}
\int_1^{1+\ell(1) d^{-1/2}} {1 \over \ell(u)} \, du
$$
which to first order as $d\to\infty$ is equal to
$$
1 + {1 \over h(\betamax)}
(\ell(1) d^{-1/2}) {1 \over \ell(1)}
\ = \
1 + {d^{-1/2} \over h(\betamax)}
\, .
$$
Simultaneously, the $Z^{(d)}_t$ process proposes to move from $2-1=1$
to $2 - [1 + d^{-1/2} / h(\betamax)] = 1 - d^{-1/2} / h(\betamax)$,
and the $W^{(d)}_t$ process proposes to move from $\wmax$ to
$$
(\wmax) - d^{-1/2} / s_1 h(\betamax)
\, .
$$
Let $A$ be the probability that the original $\beta^{(d)}(t)$
process accepts a move from $1$ to $1 + \ell(1) d^{-1/2}$.  Then
since $\beta^{(d)}(t)$ proposes to move from $1$ to $1 + \ell(1) d^{-1/2}$
with probability $1/2$, it actually
moves from $1$ to $1 + \ell(1) d^{-1/2}$ with
probability~$A/2$, otherwise it stays at~$1$.
So, correspondingly, $W_t^{(d)}$ moves
from $\wmax$ to $(\wmax) - d^{-1/2} / s_1 h(\betamax)$.
Furthermore, recall that $W_t^{(d)}$ moves
at Poisson rate~$[d \, h(\betamax)^2]$, so it moves
from $\wmax$ to $(\wmax) - d^{-1/2} / s_1 h(\betamax)$ at rate
$[d \, h(\betamax)^2] (A/2)$.
However, we instead consider a minor modification of the process
$W_t^{(d)}$ which speeds up time by a factor of~2 whenever it is at
$\wmax$, i.e.\ it moves from there at Poisson rate~$[d \, h(\betamax)^2] (A)$.
This is equivalent to the original $\beta^{(d)}(t)$ process ``reflecting''
by always
proposing a positive move from~$1$, instead of proposing either a positive
or a negative (always-rejected) move with probability~$1/2$ each.
We show in Section~\ref{sec-minormod} below that this minor modification will
not change the limiting distribution of the
$W_t^{(d)}$, and thus does not affect the proof.

Thus, to first order as $\delta \searrow 0$
[i.e., up to $o(1)$ errors], our modified process $W_t^{(d)}$
will move from $\wmax$ to $(\wmax) - d^{-1/2} / s_1 h(\betamax)$
at Poisson rate~$[d \, h(\betamax)^2] (A)$.
Hence, setting $x=\wmax = 1/s_1$, we have that
$$
{\E[f(W^{(d)}_{\delta}) \, | \, W^{(d)}_0=x] - f(x) \over {\delta}}
\ = \ 
[d \, h(\betamax)^2]
(A) \Big[ f\Big( (\wmax) - d^{-1/2} / s_1 h(\betamax) \Big) - f(x) \Big]
+ o(1)
\, .
$$
Then, taking a Taylor series expansion around $x=\wmax = 1/s_1$,
\begin{multline*}
{\E[f(W^{(d)}_{\delta}) \, | \, W^{(d)}_0=x] - f(x) \over {\delta}}
\ = \ - [d \, h(\betamax)^2] (A) [ d^{-1/2} / s_1 h(\betamax) ] f'(\wmax)
\\
\qquad\qquad\qquad\qquad
\ + \ \half [d \, h(\betamax)^2] (A) [ d^{-1/2} / s_1 h(\betamax) ]^2 f''(\wmax)
\ + \ O(d^{-1/2}) + o(1)
\\
\ = \ - [A d^{1/2} h(\betamax) / s_1] \, f'(\wmax)
\ + \ \half [A / s_1^2] \, f''(\wmax)
\ + \ O(d^{-1/2}) + o(1)
\, ,
\end{multline*}
Since $f\in\D$, we have $f'(\wmax)=0$, so the first term vanishes.
Furthermore, it is shown in \cite{TawnRR2020} that as $d\to\infty$,
$$
A \ \to \
2 \, \Phi\left( {- \ell_0 \over 2 \sqrt{r_1}} \right)
 \ = \ s_1^2
\, .
$$
Hence,
$$
{\E[f(W^{(d)}_{\delta}) \, | \, W^{(d)}_0=x] - f(x) \over {\delta}}
\ = \ 0 \ + \ \half \, [1] \, f''(\wmax) \ + \ O(d^{-1/2}) + o(1)
\, ,
$$
so that
$$
\lim_{d\to\infty}
G^{(d)}f(\wmax)
\ = \ \lim_{d\to\infty} \ \lim_{\delta \searrow 0}
{\E[f(W^{(d)}_\delta) \, | \, W^{(d)}_0=x] - f(x) \over \delta}
\ = \ \half f''(\wmax)
\ = \ G^*(\wmax)
\, ,
$$
as required.

\subsubsection{Verification of~\eqref{G0eqn}}

To prove~\eqref{G0eqn}, note that if the original inverse-temperature
process $\beta^{(d)}(t)$ proposes to move in time~1
from $\betamax$ to $\betamax - \ell(\betamax) d^{-1/2}$ in
one of the two modes (with probabilities $w_1$ and $w_2$ respectively),
then by~\eqref{Htransform} the $H^{(d)}_t$ process proposes to move
at rate~$[d \, h(\betamax)^2]$
from $1+{h(\betamax) \over h(\betamax)} = 2$ to
$$
\pm \left[ 1 + {h\big( \betamax - \ell(\betamax) d^{-1/2} \big)
		\over h(\betamax)} \right]
\ = \ \pm \left[ 2 - {\int_{\betamax-\ell(\betamax) d^{-1/2}}^{\betamax}
			{1 \over \ell(u)} \, du
		\over h(\betamax)} \right]
$$
$$
\ \approx \ \pm \left[ 2 - \left( \ell(\betamax) d^{-1/2} \right)
		{1 \over \ell(\betamax)} \right]
\ = \ \pm (2-d^{-1/2})
\, .
$$
Simultaneously,
the $Z^{(d)}_t$ process proposes to move from $2-2=0$ to
$\pm 2 - [\pm (2 - d^{-1/2})] = \pm d^{-1/2}$,
and the $W^{(d)}_t$ process proposes to move from $0$ to
either $d^{-1/2}/s_1$ or $-d^{-1/2}/s_2$.  Hence, similar to the above
(but without the minor modification),
with $x=0$
we have to first order as $\delta\searrow 0$ that
\begin{multline}\label{Gdinit}
{\E[f(W_{\delta}) \, | \, W_0=x] - f(x) \over {\delta}}
\\
\ = \ [d \, h(\betamax)^2] \, \Big( w_1 \alpha_1
\left[f \left( d^{-1/2} / s_1 \right) - f(0) \right]
\, + \,
w_2 \alpha_2
\left[f \left(- d^{-1/2} / s_2 \right) - f(0) \right]
\Big)
+ o(1)
\, ,
\end{multline}
where $\alpha_i$ is the acceptance probability for the original
process to accept a proposal
to increase the inverse-temperature
from $\betamax$ to $\betamax - \ell(\betamax) d^{-1/2}$ in mode~i.
Now, the argument in~\cite{TawnRR2020} shows that as $d\to\infty$ we have
$$
\alpha_i \ \to \
2 \, \Phi\left( {- \ell_0 \over 2 \sqrt{r_i}} \right)
\ = \ s_i^2
\, ,
\qquad i=1,2
\, .
$$
Hence, taking a Taylor series expansion around $x=0$,
we obtain from~\eqref{Gdinit} that
\begin{multline*}
{\E[f(W_{\delta}) \, | \, W_0=x] - f(x) \over {\delta}}
\\
\ = \
d \, w_1 s_1^2
\left(d^{-1/2} / s_1 \right)
f'^+(0)
+ \ \frac{1}{2} d \, w_1 s_1^2
\left(d^{-1/2} / s_1 \right)^2
f''^+(0)
\ + \ O(d \, d^{-3/2}) + o(1)
\\
- \ d \, w_2 s_2^2
\ \left(d^{-1/2} / s_2 \right)
f'^-(0)
+ \ \frac{1}{2} d \, w_2 s_2^2
\left(d^{-1/2} / s_2 \right)^2
f''^-(0)
+ \ O(d \, d^{-3/2}) + o(1)
\\
\ = \
d^{1/2} [ w_1 s_1 f'^+(0) - w_2 s_2 f'^-(0) ]
\ + \ \frac{1}{2} [ w_1 f''^+(0) + w_2 f''^-(0) ]
\ + \ O(d^{-1/2}) + o(1)
\, .
\end{multline*}
Now, by the definition of $f\in \D$,
$w_1 s_1 f'^+(0) - w_2 s_2 f'^-(0) = 0$,
and $w_1 f''^+(0) + w_2 f''^-(0) = (w_1+w_2) f''(0) = f''(0)$.
Hence, we obtain finally that
$$
{\E[f(W_{\delta}) \, | \, W_0=x] - f(x) \over {\delta}}
\ = \
\frac{1}{2}f''(0) + O(d^{-1/2}) + o(1)
\, ,
$$
so that
$$
\lim_{d\to\infty} G^{(d)}f(0)
\ = \ \lim_{d\to\infty} \ \lim_{\delta \searrow 0}
{\E[f(W^{(d)}_\delta) \, | \, W^{(d)}_0=x] - f(x) \over \delta}
\ = \ \half f''(0)
\ = \ G^*(0)
\, .
$$
This establishes~\eqref{G0eqn}, and hence
completes the proof of Theorem~\ref{Theorem:skewBM}.

\begin{remark}
{\rm
Because of our transformations converting $\beta^{(d)}(t)$ to $W^{(d)}_t$,
the limiting process $W$ in Theorem~\ref{Theorem:skewBM}
was skew Brownian motion on a continuous interval.
It is also possible to consider diffusions directly associated with a
discrete graph, see e.g.~\cite{Freidlin2001}.
}
\end{remark}


\section{Appendix: Modified Processes and Occupation Times}
\label{sec-minormod}

\def\Whatdt{W_{\tau_d(t)}^{(d)}}
\def\Whatd0{W_0^{(d)}}

Recall that the proof of~\eqref{Gwmaxeqn}
in Section~\ref{sec-1314} above was actually for a minor modification
of the process $W_t^{(d)}$, which speeds up time by a factor of~2 whenever
it is in the state $\wmax$.  We now argue that this minor modification
does not affect the limiting distribution.  Indeed,
since the modification corresponds to adjusting the rate of time,
we can write the modified process as $\widehat{W}_t^{(d)}
\equiv \Whatdt$, where $\tau_d(t)$ is the time
scale including the occasional speedups.
Clearly $\lim_{t \searrow 0} \tau_d(t) = 0$.
Also, it follows from Proposition~\ref{occtime} below
that the fraction of time that the original process spends at $\wmax$
converges to~0 as $d\to\infty$.  This implies that
$\lim_{d\to\infty} (\tau_d(t) / t) = 1$.
Since our process $W_t^{(d)}$ is continuous, this means that
$\lim_{d\to\infty} |f(\Whatdt) - f(W_t^{(d)})| = 0$.
That is, the two processes have the same limiting behaviour
as $d\to\infty$.  So, the diffusion limit is
not affected by making our minor modification as above.


It remains to state and prove Proposition~\ref{occtime}.  We begin with
a result about limiting probabilities for reflecting simple
symmetric random walk.


\begin{proposition}\label{refrw}
Let $\{Y_n\}$ be reflecting simple symmetric random walk
on the state space $\{0,1,2,\ldots,m\}$,
i.e.\ a discrete-time birth-death Markov chain
with transition probabilities
$p_{i,i+1}=p_{i,i-1}=1/2$ for $1 \le i \le m-1$,
and $p_{0,1}=p_{m,m-1}=1$.  Then
for all $m\in\IN$ and all sufficiently large $n\in\IN$,
$\P(Y_n = 0) \, \le \, (2/\sqrt{n}) + (1/m)$.
Hence, $\lim\limits_{n,m\to\infty} \P(Y_n = 0) \, = \, 0$.
\end{proposition}

\begin{figure}[ht]
\centerline{\includegraphics[width=14cm]{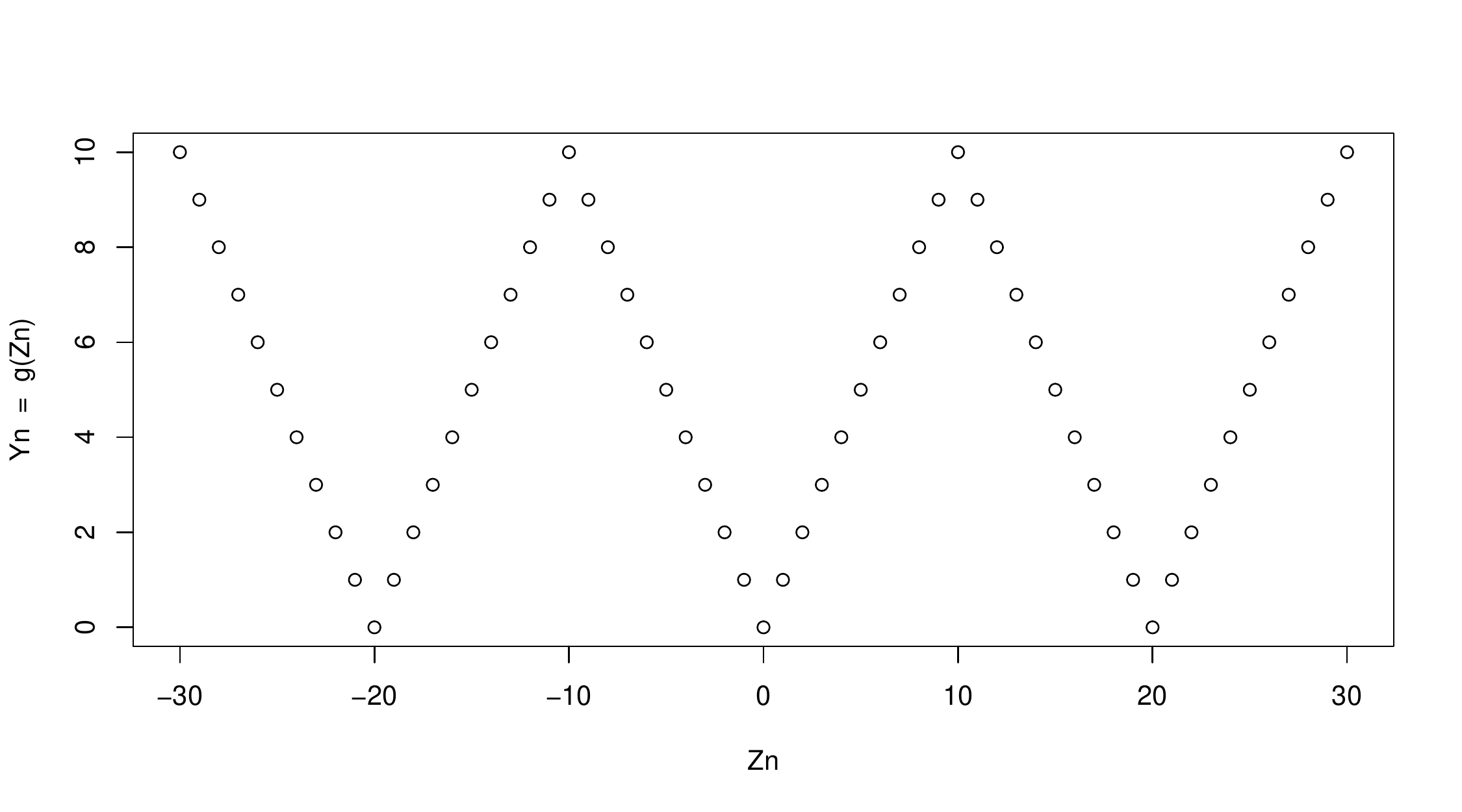}}
\caption{The lifting transformation function ``$g$'' (when $m=10$).}
\label{fig-g-function}
\end{figure}

\begin{proof}
We condition on $Y_0=y$; the general case then follows by taking
expectation with respect to $Y_0$.
We ``lift'' $\{Y_n\}$ to $\IZ$ by writing
$Y_n = g(Z_n)$, where $\{Z_n\}$ is simple symmetric
random walk on {\it all} the integers $\IZ$, and
$g(z) = \min_j |z-2jm|$ (see Figure~\ref{fig-g-function}).
Then
$$
\P_y[Y_n=0]
\ = \ \P_y[g(Z_n)=0]
\ = \ \sum_{j\in\IZ} \P_y[Z_n=2jm]
$$
$$
\ = \ \sum_{j\in\IZ} \P_y\Big[\Binnh = \n2 + \y2 + jm\Big]
\ = \ \sum_{j\in\IZ} h\left(\n2+\y2+jm\right)
\, ,
$$
where $h(k) = \P[\Binnh = k]$.  Now, $h$ is maximised
when $k=n/2$ (or $(n\pm 1)/2$ if $n$ is odd), and decreases monotonically
on either side of that.  Hence, find $j_*\in\IN$ with
$\y2+(j_*-1)m < 0 \le \y2+j_*m$.
It follows from Stirling's Approximation (see e.g.~\cite{maxbinom.tex})
that to first order as $n,k,n-k\to\infty$,
$$
\P[\Binnh = k]
\ \le \ e^{-2n[\half-{k \over n}]^2} \, \sqrt{1 / 2 \pi k [1-(k/n)]}
\, ,
$$
so in particular
$$
h\left(\n2+\y2+j_*m\right)
\ \le \ \sqrt{2/\pi n} + o_n(1)
\ \le \ 1/\sqrt{n}
$$
for all sufficiently large $n$, and similarly for
$h\left(\n2+\y2+(j_*-1)m\right)$.
Then, by monotonicity, we have for $j>j_*$ that
\begin{multline*}
h\left(\n2+\y2+jm\right) \ \le \
{1 \over m} \Big[ h\left(\n2+\y2+(j-1)m+1\right)+h\left(\n2+\y2+(j-1)m+2\right)
\\
					+\ldots+h\left(\n2+\y2+jm\right) \Big]
\, .
\end{multline*}

\noindent
Hence,
$$
\sum_{j > j_*} h\left(\n2+\y2+jm\right)
\ \le \ {1 \over m} \Big[ h\left(\n2+\y2+1\right)+h\left(\n2+\y2+2\right)
			+h\left(\n2+\y2+3\right)+\ldots \Big]
\, .
$$
But $\sum_k h(k) = 1$, so by symmetry
$\sum_{k>n/2} h(k) \le 1/2$, and so
$$
h\left(\n2+\y2+1\right)+h\left(\n2+\y2+2\right)+h\left(\n2+\y2+3\right)
						+\ldots \le 1/2
\, .
$$
Thus,
$$
\sum_{j > j_*} h\left(\n2+\y2+jm\right)
\ \le \ {1 \over 2m}
\, .
$$
Similarly,
$$
\sum_{j < j_*-1} h\left(\n2+\y2+jm\right)
\ \le \ {1 \over 2m}
\, .
$$
Therefore, for all sufficiently large $n$,
$$
\sum_{j\in\IZ} h\left(\n2+\y2+jm\right)
\ \le \
(1/\sqrt{n}) + (1/\sqrt{n}) + {1 \over 2m} + {1 \over 2m}
\ = \ (2/\sqrt{n}) + (1 / m)
\, ,
$$
as claimed.
\end{proof}


\medskip
\begin{remark}
{\rm
Similar arguments show that
$\lim\limits_{n,m\to\infty} \P(Y_n = z) \, = \, 0$ for any
fixed number $z\in\IN$, by replacing ``$Z_n=2jm$'' by ``$Z_n=2jm+z$'',
and ``$\n2+\y2$'' by ``$\n2+\y2-\z2$'', throughout the proof,
though we do not use that fact here.
}
\end{remark}

\medskip
\begin{corollary}\label{refcor}
Let $\{Y_n\}$ be as in Proposition~\ref{refrw}.
Let $N_0=\#\{i : 0 \le i \le n-1, \ Y_i=0\}$ be the occupation time of the
state~0 before time~$n$.  Then as $n,m\to\infty$,
the average occupation time $N_0/n$ converges to 0 in probability.
\end{corollary}

\begin{proof}
Let $I_i = \one_{Y_i=0}$ be the indicator function of the event $Y_i=0$.
Then by Proposition~\ref{refrw},
$\lim_{n,m\to\infty} \E[I_n] = \lim_{n,m\to\infty} \P[Y_n=0] = 0$.
Hence, using the theory of Ces\`aro sums,
\begin{multline*}
\lim_{n,m\to\infty} \E[N_0/n]
\ = \ \lim_{n,m\to\infty} \E\Big[ \sum_{i=0}^{n-1} I_i \Big] / n
\ = \ \lim_{n,m\to\infty} {1 \over n} \, \sum_{i=0}^{n-1} \E[I_i]
\ = \ \lim_{n,m\to\infty} \E[I_{n}]
\ = \ 0
\, .
\end{multline*}
Hence, by Markov's inequality, since $N_0/n \ge 0$, for any $\epsilon>0$
we have
$$
\lim_{n,m\to\infty} \P[(N_0/n) > \epsilon]
\ \le \ \lim_{n,m\to\infty} \E[N_0/n] / \epsilon
\ = \ 0
\, ,
$$
so that $N_0/n \to 0$ in probability, as claimed.
\end{proof}


\begin{proposition}\label{occtime}
Let $\{X_n\}$ be a discrete-time birth-death Markov chain
on the state space $\{0,1,2,\ldots,m\}$,
with transition probabilities satisfying that
$p_{i,j}=0$ whenever $|j-i| \ge 2$,
$p_{i,i+1}=p_{i,i-1}$ for all $1 \le i \le m-1$,
and $p_{i,i} \le 1-a$ for some fixed constant $a>0$.
Let $N_0=\#\{i : 0 \le i \le n-1, \ X_i=0\}$.
Then as $n,m\to\infty$, $N_0/n$ converges to 0 in probability.
\end{proposition}

\begin{proof}
Let $\{J_k\}$ be the {\it jump chain} of $\{X_n\}$,
i.e.\ the Markov chain which copies $\{X_n\}$ except omitting immediate
repetitions of the same state, and let
$\{M_k\}$ count the number of repetitions.
[For example, if the original chain $\{X_n\}$ began
$
\{X_n\} = ( a, b, b, b, a, a, c, c, c, c, d, d, a, \ldots )
\, ,
$
then the jump chain $\{J_k\}$ would begin
$
\{J_k\} = ( a, b, a, c, d, a, \ldots )
\, ,
$
and the corresponding multiplicity list $\{M_k\}$ would begin
$
\{M_k\} = ( 1, 3, 2, 4, 2, \ldots )
\, .
$]
Then the assumptions imply that $\{J_k\}$ has the transition
probabilities of reflecting simple symmetric random walk, as
in Proposition~\ref{refrw} and Corollary~\ref{refcor} above.

Now, let $K(n)$ be the smallest integer with
$M_1+\ldots+M_{K(n)} \ge n$.
Given $J_k$, the random variable $M_k$ has the Geometric$(1-p_{J_k J_k})$
distribution, so it is stochastically bounded above by the Geometric$(a)$
distribution, from which it follows that $\lim_{n\to\infty} K(n)
= \infty$ w.p.~1.
Let $C_s=\#\{i : 0 \le i \le K(n), \ J_i=s\}$.
Then Corollary~\ref{refcor} implies that $\lim_{n,m\to\infty} (C_0/K(n))
= 0$.
On the other hand,
$N_0$ is $\le$ a sum of $C_0$ independent Geometric$(1-p_{00})$
random variables, so $\E[N_0 \, | \, C_0] = C_0/(1-p_{00}) \le C_0/a$,
and $\P[N_0 > 2C_0/a \, | \, C_0] \to 0$ as $n\to\infty$.
Also, $M_1+\ldots+M_{K(n)-1} \le n$, and each $M_i \ge 1$, so
$n \ge K(n)-1$.  We therefore conclude that
$$
\lim_{n,m\to\infty} {N_0 \over n}
\ \le \ \lim_{n,m\to\infty} {2C_0/a \over K(n)-1}
\ = \ ({2 / a}) \, \lim_{n,m\to\infty} {C_0 \over K(n)}
\ = \ 0
\, ,
$$
as claimed.
\end{proof}



\begin{remark}\rm
It might be possible to instead obtain the conclusion of
Proposition~\ref{occtime} via the Ergodic Theorem or
the Markov chain Law of Large Numbers, which states that for fixed $m$
the average occupation time $N_0/n$ will converge to the stationary
measure of the process at state~0.  However, this would require
conditions and bounds on the sequence of stationary measures as
$m\to\infty$, so it would not be trivial
(nor provide the explicit bound of Proposition~\ref{refrw}),
and we instead use the more
direct and quantitative method described herein.
\gcheck
\end{remark}

\bigskip\noindent\bf Acknowledgements. \rm
We thank Alex Mijatovic and Neal Madras for very helpful comments related
to Section~\ref{sec-minormod} herein, and thank the editor and referees
for very insightful suggestions which greatly improved the manuscript.
This research was partially supported by EPSRC grants EP/R018561/1
and EP/R034710/1 to GOR, and NSERC grant RGPIN-2019-04142  to JSR.

\medskip

\bibliographystyle{plain}
\bibliography{tawnbib}

\end{document}